\documentclass[11pt]{amsart}
\usepackage{amsmath, amsthm, mathabx,amscd, amsfonts, amssymb, hyperref, mathrsfs, textcomp, epsfig, a4wide, graphicx, IEEEtrantools, tikz, verbatim, xypic}
\usepackage[mathscr]{euscript}

\usetikzlibrary{shapes,arrows}
\usetikzlibrary{matrix,decorations.pathreplacing,angles,quotes}

\newtheorem{thm}{Theorem}[section]

\newtheorem{claim}[thm]{Claim}

\theoremstyle{definition}

\theoremstyle{remark}
\newtheorem{remark}{Remark}[section]

     \RequirePackage{rotating}                   
    \def\HSt{%
       \setbox0=\hbox{$\widehat{\mathit{HS}}$}
       \setbox1=\hbox{$\mathit{HS}$}
       \dimen0=1.1\ht0
       \advance\dimen0 by 1.17\ht1
       \smash{\mskip2mu\raise\dimen0\rlap{%
          \begin{turn}{180}
              {$\widehat{\phantom{\mathit{HS}}}$}
           \end{turn}} \mskip-2mu    
                \mathit{HS}
    }{\vphantom{\widehat{\mathit{HS}}}}{}}

     \RequirePackage{rotating}                   
    \def\HMt{%
       \setbox0=\hbox{$\widehat{\mathit{HM}}$}
       \setbox1=\hbox{$\mathit{HM}$}
       \dimen0=1.1\ht0
       \advance\dimen0 by 1.17\ht1
       \smash{\mskip2mu\raise\dimen0\rlap{%
          \begin{turn}{180}
              {$\widehat{\phantom{\mathit{HM}}}$}
           \end{turn}} \mskip-2mu    
                \mathit{HM}
    }{\vphantom{\widehat{\mathit{HM}}}}{}}

\newcommand{\vol}{\mathrm{vol}}

\newcommand{\R}{\mathbb{R}}
\newcommand{\SW}{\mathsf{SW}}

\begin{document}

\title[Monopoles and the dodecahedron]{Monopole Floer homology, eigenform multiplicities and the Seifert-Weber dodecahedral space} 

\author{Francesco Lin}
\address{Department of Mathematics, Columbia University} 
\email{flin@math.columbia.edu}

\author{Michael Lipnowski}
\address{Department of Mathematics and Statistics, McGill University} 
\email{michael.lipnowski@mcgill.ca}
\begin{abstract}
We show that the Seifert-Weber dodecahedral space $\SW$ is an $L$-space. The proof builds on our work relating Floer homology and spectral geometry of hyperbolic three-manifolds. A direct application of our previous techniques runs into difficulties arising from the computational complexity of the problem. We overcome this by exploiting the large symmetry group and the arithmetic and tetrahedral group structure of $\SW$ to prove that small eigenvalues on coexact $1$-forms must have large multiplicity.
\end{abstract}
\maketitle

The Seifert-Weber dodecahedral space $\SW$ is obtained by identifying opposite faces of a dodecahedron by a $3/10$ full turn \cite{Thu}; it was one of the first examples of closed hyperbolic three-manifold to be discovered \cite{WS}. Despite its very simple description, it is a quite complicated space from the point of view of three-dimensional topology. For example, the conjecture of Thurston from $1980$ that $\SW$ is not Haken has been considered a benchmark problem in computational topology and took $30$ years to settle \cite{BRT}.
\\
\par
In the present paper, we look at $\SW$ from the point of view of monopole Floer homology \cite{KM}. Recall that $H_1(\SW)= \left(\mathbb{Z} / 5\mathbb{Z} \right)^3$, so that $\SW$ is a rational homology sphere.
\begin{thm}\label{main}
The Seifert-Weber dodecahedral space $\SW$ is an $L$-space, i.e. its reduced Floer homology $\mathit{HM}_*(\SW)$ vanishes.
\end{thm}{More is true, in fact: our proof will show that, for all $\mathrm{spin}^c$-structures on $\SW$, small perturbations of the Seiberg-Witten equations on $M$ admit no irreducible solutions, and therefore $\SW$ is a \textit{minimal $L$-space} in the sense of \cite{LL}.
\\
\par
As a direct consequence of Theorem \ref{main}, we obtain that the Seifert-Weber dodecahedral space does not admit coorientable taut foliations \cite{KMOS}. Furthermore, as $\SW$ is also an arithmetic hyperbolic three-manifold of the simplest type with $H_1(\SW,\mathbb{Z}/2\mathbb{Z})=0$ (see Remark \ref{simplest}), the construction of \cite{AL} can be directly adapted to provide more examples of hyperbolic $4$-manifolds with vanishing Seiberg-Witten invariants.
\\
\par
Our approach to Theorem \ref{main} builds on the ideas of our previous work \cite{LL}. There, we showed that a hyperbolic rational homology sphere $Y$ for which the first eigenvalue on coexact $1$-forms $\lambda_1^*$ is strictly larger than $2$ is an $L$-space. We then developed numerical techniques (based on the Selberg trace formula) to provide explicit lower bounds on $\lambda_1^*$ in terms of the closed geodesics of $Y$. More specifically, taking as input the list of complex lengths of geodesics with length at most $R$ (as computed for example by SnapPy \cite{CDGW}), we determine a function $J_{R,t}(Y)$ which is an upper bound to the multiplicity of $t^2$ as an eigenvalue of $\Delta$ on coexact $1$-forms. In particular, if $J_{R,t}<1$ then $t^2$ is not an eigenvalue; using this, we showed that several manifolds with small volume ($\leq 2.03$) have $\lambda_1^*>2$, and are therefore $L$-spaces.
\\
\par
One can try to apply the same approach to $\SW$, which has significantly larger volume $\approx 11.119$. Using SnapPy, we computed the length spectrum up to cutoff $R=8$ in about $5$ hours. The function $J_{R,t}(\SW)$ for $R=8$ can be found in Figure \ref{SWgraph}; unfortunately, it only proves that $\lambda_1^*> 1.9188 \ldots$. 
\begin{figure}
  \includegraphics[width=0.8\linewidth]{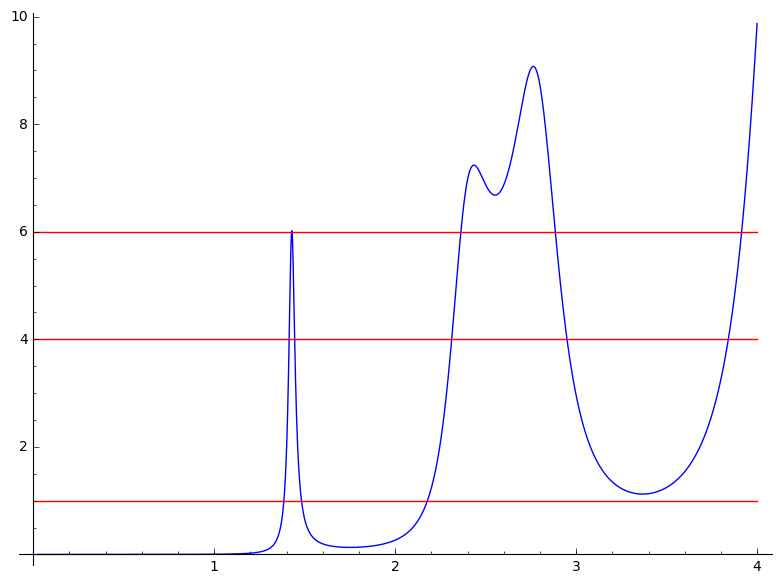}
  \caption{The graph of $t \mapsto J_{8,t}(\SW)$ for $t \in [0,4]$.}
  \label{SWgraph}
\end{figure}
On the other hand, the graph of $J_{8,t}$ is peaked just barely above height 6 in the narrow interval $[1.427877\ldots, 1.430337\ldots]$; this strongly suggests that
\begin{equation*}
\lambda_1^\ast \in [(1.427877\ldots)^2, (1.430337\ldots)^2] = [2.03883\ldots,2.04586\ldots]
\end{equation*}
and that the corresponding eigenspace has dimension $6$. As we expect $J_{R,t}$ to approximate better and better the indicator function of the spectrum (with multiplicities) for large $R$, one could in principle prove that $\lambda_1^*>2$ by showing that $J_{R,t}<1$ for $t<\sqrt{2}$ by computing the length spectrum for some larger value of $R$. In practice, this would require quite an extensive computation, because the amount of time to compute the length spectrum at cutoff $R$ grows at least exponentially in $R$, see Table \ref{table1}.

\begin{table}[h!]
 \begin{tabular}{|c| c| c|} 
 \hline
 Length cutoff $R$ & $\sqrt{2} - \text{height 1 crossing point}$ & Running time in CPU seconds \\ 
 \hline\hline
6.0 & 0.207732 \ldots &  12\\ 
 \hline
6.5 &  0.152766\ldots &  61\\ 
 \hline
7.0 &  0.096385\ldots &  348\\ 
 \hline
7.5 &  0.051723\ldots &  2255\\ 
 \hline
8.0 &  0.028979\ldots &  19112\\ 
 \hline
\end{tabular}
\caption{Differences between $\sqrt{2}$ and point where the graphs of $t \mapsto J_{R,t}$ cross height 1 for various length cutoffs $R$. Assuming the actual value of $\sqrt{\lambda_1^\ast}$ to be about $1.428..$, this suggest that a computation of length spectrum at cutoff $R=9.5$ could prove $\lambda^*_1>2$. Being quite optimistic (e.g. assuming that memory limitations do not affect the running time), such a computation would take at least several months. Here we used an Intel Core i7 2.7GHz with 10GB of allocated memory.}\label{table1}
\end{table}
\begin{remark}It should be pointed out that the computations for $\SW$ are \textit{extremely} fast (even though not enough for our purposes). For example, the computation at cutoff $R=6$ only took $12$ seconds, while for most of the other three-manifolds we tested before it took around $15-20$ minutes.
\end{remark}

We instead took a more conceptual approach.  Our main result is the following: 
\begin{claim}\label{multprop}
Any eigenvalue $\lambda^*\leq 64$ of the Hodge Laplacian on coexact $1$-forms on $\SW$ has multiplicity at least $4$.
\end{claim}
From this, we can prove Theorem \ref{main} by looking again at the function $J_{8,t}(\SW)$.
\begin{proof}[Proof of Theorem \ref{main}].
We have that $J_{8,t}(\SW)< 4$ for $t\leq 1.414380\ldots$ (see Figure \ref{SWgraphzoom}). This implies that $\lambda_1^*> (1.414380\ldots)^2 = 2.0004717\ldots > 2$, and Theorem \ref{main} follows from Theorem $0.3$ of \cite{LL}.
\end{proof}

\begin{figure}
  \includegraphics[width=0.8\linewidth]{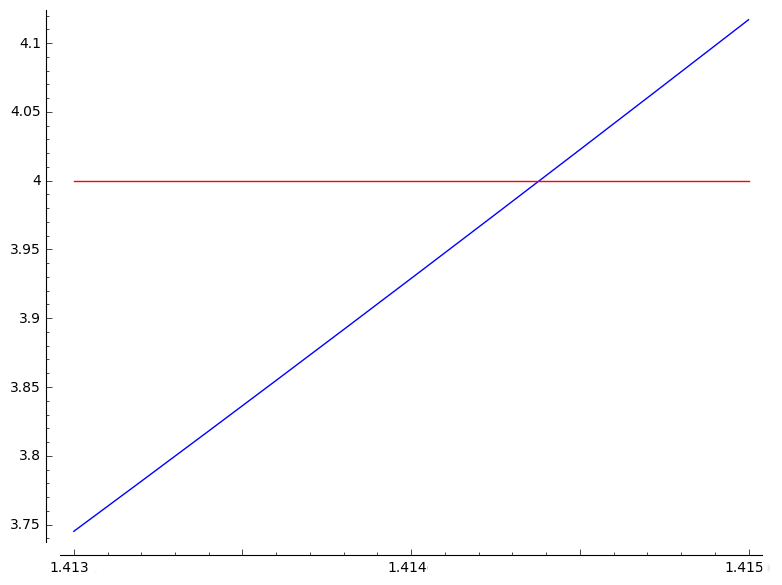}
  \caption{The graph of $t \mapsto J_{8,t}(\SW)$ for $t \in [1.413,1.415]$.}
  \label{SWgraphzoom}
\end{figure}
The rest of this paper is dedicated to the proof of Claim \ref{multprop}. Determining eigenvalue multiplicities is in general a very delicate problem, especially in the context of numerical computations.  There are two key observations about $\SW$ underlying Claim \ref{multprop}:
\begin{enumerate}
\item $\SW$ is a very symmetric manifold.  In particular, its isometry group is isomorphic to the symmetric group $S_5$ \cite{Med}. Here the alternating subgroup $A_5$ corresponds to the orientation preserving isometries of the dodecahedron.
\item while the orbifold $\SW/A_5$ is not implemented in SnapPy (which currently only has infrastructure to handle orbifolds with cyclic singularities), it admits a very explicit arithmetic description; in particular, one can use number theoretic techniques to compute its length spectrum (see Chapter $30$ of \cite{Voi}).
\end{enumerate}
By generalizing the techniques of \cite{LL} to the case of orbifolds, we can use the computation of the length spectrum (together with some additional geometric data) to show that the orbifold $\SW/A_5$ satisfies $\lambda_1^\ast >64$; this implies that the $\lambda^\ast$-eigenspace of $\SW$ for every $\lambda^\ast \leq64$, viewed as a representation of the isometry group $S_5,$ does not contain copies of the trivial representation of $A_5.$  From this we will be able to conclude Claim \ref{multprop} via the classification of irreducible representations of $S_5$.
\\
\par
Our approach to Theorem \ref{main} is based on many of the beautiful geometric and arithmetic properties of $\SW$. One would expect that the same result can be achieved by other means, given that there is an algorithmic (yet impractical) way to determine whether a given rational homology sphere is an $L$-space \cite{SW}. On the other hand, even though $\SW$ is a $5$-fold branched cover of the Whitehead link\footnote{More precisely, it is the one corresponding to the homomorphism $H_1(S^3\setminus L)\rightarrow \mathbb{Z}/5\mathbb{Z}$ sending one meridian to $1$ and the other to $2$.}, it does not admit a simple surgery description and is not the double branched cover of any link in $S^3$ (see Remark \ref{branchedS3}), so at least the most efficient computational tools available seem not to be directly applicable to it.
\\
\par
\textit{Plan of the paper.} In \S \ref{seifwebgeom} we discuss several geometric and arithmetic properties of $\SW$ which will be relevant for our purposes. In \S \ref{lower}, we generalize our techniques from \cite{LL} to the case of orbifolds and apply it to the case of $\SW/A_5$. Finally, in \S \ref{proofmult} we prove Claim \ref{multprop}, and and in \S\ref{icosahedron} we discuss a closely related tetrahedral orbifold. 
\\
\par
\textit{Acknowledgements.} The authors would like to thank Ian Agol for suggesting that the Seifert-Weber dodecahedral space could be an interesting example to test their techniques, and Aurel Page for sharing his code to compute length spectra of arithmetic orbifolds. The first author was partially supported by NSF grand DMS-1807242 and an Alfred P. Sloan fellowship.
\vspace{0.5cm}
\section{The geometry and arithmetic of $\SW$}\label{seifwebgeom}
\subsection{The isometry group.}We follow closely the discussion of \cite{Med}, to which we refer for additional details. As mentioned in the introduction, $\SW$ is obtained by identifying opposite faces of a dodecahedron by a $3/10$ full turn. Under this identification, all vertices of the dodecahedron get identified and the $30$ edges get identified in groups of five. The hyperbolic metric is realized by considering the regular dodecahedron $\mathsf{D}$ in $\mathbb{H}^3$ with dihedral angles of $2\pi/5=72^{\degree}$. The barycentric subdivision of $\mathsf{D}$ is made of $120$ copies of the the tetrahedron $\mathsf{T}$ with totally geodesic faces in Figure \ref{tetrahedron}.
\begin{figure}
  \centering
\def\svgwidth{\textwidth}
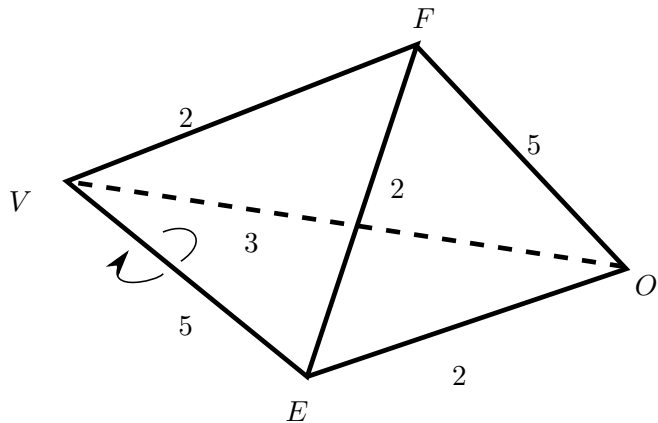
\caption{A schematic picture of the tetrahedron $\mathsf{T}$. An edge labeled $n$ has dihedral angle $\pi/n$. The vertices $O$, $V$, $E$ and $F$ are respectively the center of the dodecahedron $\mathsf{D}$, one of its vertices, the center of an edge and the center of a face. In Coxeter notation this is the tetrahedron $[5,3,5]$. The arrow denotes the orientation for the order $5$ rotation $c$.}
\label{tetrahedron}
\end{figure}

Denote by $\Gamma$ the group generated by reflections across the faces of $\mathsf{T}$, and by $\Gamma^+<\Gamma$ the index two subgroup consisting of orientation-preserving isometries. 
This has presentation 
\begin{equation}\label{prestetr}
\Gamma^+=\langle a,b,c\lvert a^2=b^2=c^5=(bc)^2=(ca)^3=(ab)^5=1\rangle
\end{equation}
where $a$, $b$ and $c$ are the rotations around the axes $VF$, $EF$ and $EV$ of angles $\pi$, $\pi$ and $2 \pi/5$ (the latter with orientation as in Figure \ref{tetrahedron}); see Section $4.7$ of \cite{MacReid}. In terms of reflections, defining for a vertex $P$ of $\mathsf{T}$ the reflection $R_P$ across the face opposite to $P$, we have
\begin{equation*}
a=R_ER_O,\quad b=R_VR_O,\quad c=R_FR_O.
\end{equation*}
Following \cite{Med}, the surjective homomorphism
\begin{align*}
\varphi:&\Gamma^+\twoheadrightarrow A_5\\
a\mapsto (23)(45)\quad b&\mapsto (12)(35)\quad c\mapsto (12345)
\end{align*}
has kernel identified with $\pi_1(\SW)$. In particular, we see that there is a natural action of $A_5$ on $\SW$ by isometries (corresponding to the isometry group of $\mathsf{D}$, which is isomorphic to $A_5$), and the quotient orbifold $\SW/A_5$ has fundamental group $\Gamma^+$. Furthermore, the fundamental domain for the action of $\Gamma^+$ on $\mathbb{H}^3$ is given by doubling $\mathsf{T}$ along any face; looking at the identifications, we see that the quotient $\SW/A_5$ is homeomorphic to $S^3$, and the orbifold locus is described again by Figure \ref{tetrahedron}, when thought of as a labeled trivalent graph in $S^3$. In particular the isotropy groups of the vertices $V$, $O$ are isomorphic to $A_5$ and the isotropy groups of $E$, $F$ are isomorphic to the dihedral group with $10$ elements $D_{10}$.
\\
\par
The tetrahedron $\mathsf{T}$ admits an incidence and edge label-preserving symmetry sending $V$ to $O$ and $E$ to $F$; because there is a unique hyperbolic tetrahedron with given dihedral angles, this symmetry is realized by an isometry $\iota$. Geometrically, $\iota$ is the rotation by $\pi$ along the geodesic connecting the midpoints of the segments $EF$ and $VO$ (such a geodesic is necessarily orthogonal to the edges at the endpoints). Since $\iota$ maps faces to faces, it normalizes $\Gamma$ and hence the orientation-preserving subgroup $\Gamma^+.$
Thus, the subgroup $\Lambda$ of $\mathrm{Isom}^+(\mathbb{H}^3)$ generated by $\Gamma^+$ and $\iota$ is the semidirect product of $\Gamma^+$ and the order two subgroup generated by $\iota.$  Furthermore, $\varphi$ can be extended to a homomorphism
\begin{equation*}
\tilde{\varphi}:\Lambda\rightarrow S_5
\end{equation*}
by sending $\iota$ to $x = (12)$. Indeed, extending to $\Lambda$ is equivalent to the relation $\varphi(\iota g \iota^{-1}) = x \varphi(g) x^{-1}$ for all $g \in \Gamma^+.$  It is readily checked that this relation holds for every $g$ among the generators $a,b,c$ of $\Gamma^+,$\footnote{For example, because $\iota$ conjugates $R_P$ to $R_{\iota(P)}$, it conjugates $a = R_E R_O$ to $R_F R_V = R_F R_0 R_0 R_V = cb^{-1}.$  So indeed,
$$\varphi(\iota a \iota^{-1}) = \varphi( cb^{-1}) = (1 3) (4 5) = x \varphi(a) x^{-1}.$$} which suffices.

In particular, $\iota$ normalizes the kernel of $\varphi$, and hence induces an isometry of $\SW$. This isometry switches $O$ and $V$, and will be referred to as the \textit{inside-out} isometry. The main theorem of \cite{Med} is then the following.
\begin{thm}[\cite{Med}]
The isometry group of $\SW$ is isomorphic to $S_5$, and is generated by $\iota$ and the orientation-preserving isometry group of the dodecahedron $\cong A_5$.
\end{thm}
\begin{remark}\label{branchedS3}
In \cite{Med}, the author also determines the action of $S_5$ on the first homology $H_1(\SW)$.  In particular, for the order $2$ elements in $S_5$ the quotient has non trivial $H_1$, and is therefore not $S^3$. By geometrization, $\SW$ is not the branched double cover of a link in $S^3$.
\end{remark}
\begin{remark}\label{simplest}
From our description it is clear that $\pi_1(\SW)$ is commensurable to a tetrahedral group, and is therefore arithmetic of the simplest type by \cite{MRtet}.
\end{remark}

\vspace{0.3cm}
\subsection{Covolume of centralizers.}\label{centr} For later applications, for a non-trivial element $\gamma\in\Gamma^+$ we need to understand the volume $\vol(\Gamma^+_{\gamma}\setminus G_{\gamma})$, where $G_{\gamma}$ and $\Gamma^+_{\gamma}$ are the centralizers of $\gamma$ in $G=\mathrm{PSL}_2(\mathbb{C})$ and $\Gamma^+$ respectively. In closed hyperbolic manifold groups, all elements are hyperbolic and the centralizer of $\gamma_0^n$ with $\gamma_0$ primitive is the cyclic group generated by $\gamma_0$, and therefore the quantity of interest is simply the translation length $\ell(\gamma_0)$. For the closed hyperbolic orbifold group $\Gamma^+,$ we need to consider two special classes of elements:
\begin{itemize}
\item elliptic elements;
\item \textit{bad} hyperbolic elements, i.e. hyperbolic elements whose axis is the fixed axis for some elliptic element.
\end{itemize}
To determine these quantities, we will refer to the picture of $\SW/A_5$ in Figure \ref{tetrahedron} (considered again as $S^3$ with orbifold locus).
\par
Up to conjugacy, the elliptic elements of order $5$ are rotations around $FO$ and $VE$.
As these two rotations are exchanged by the inside-out isometry, we need only perform the computations on the non-trivial conjugacy classes of powers of $\gamma_f$, a rotation of $2\pi/5$ around the geodesic $f$ in $\mathbb{H}^3$ connecting $F$ and $O$. Notice that $\gamma_f$ and $\gamma_f^{-1}$ are conjugate in $\Gamma^+$ because in the isotropy group of $F$, which is isomorphic to $D_{10}$, every order two element is conjugate to its inverse.
Therefore in $\Gamma^+$ there are exactly $4$ conjugacy classes of order $5$ elliptics, namely $\gamma_f$, $\gamma_f^2$ and their conjugates under $\iota$. Now, $G_{\gamma_f}$ consists of the elliptic elements with axis $f$ and the hyperbolic elements with axis $f$ which preserve the endpoints of $f$. Furthermore, $\Gamma^+_{\gamma_f}$ is the abelian group generated by $\gamma_f$ and a primitive hyperbolic element $h_f$ in $\Gamma^+$ with axis $f$ (and which preserves endpoints). The image of $f$ in $\SW/A_5$ is $|OF|$, a geodesic of mirrored-arc type, and therefore the translation length of $h_f$ is $2|OF|$.
This quantity can be determined by looking at the geometry of the triangle with vertices $E$, $F$ and $O$ as follows. Direct geometric considerations with the dodecahedron $\mathsf{D}$ show that
\begin{equation*}
\angle OFE=\pi/2, \qquad \angle OEF=\pi/5.
\end{equation*}
The angle $\angle EOF$ can be computed via formulas of spherical trigonometry using the fact that a small sphere centered at $O$ intersects $\mathsf{T}$ in a geodesic triangle with angles $\pi/2$, $\pi/3$ and $\pi/5$. We have
\begin{equation*}
\angle EOF=\arctan\left(\frac{\sqrt{5}-1}{2}\right).
\end{equation*}
Since $OFE$ spans a geodesic hyperbolic right triangle, we obtain
\begin{equation*}
\cosh|OF|=\frac{\cos(\angle OEF)}{\sin(\angle EOF)}=1.5388\dots
\end{equation*}
so that $|OF|=0.9963...$ and 
\begin{equation*}
\vol(\Gamma^+_{\gamma_f}\setminus G_{\gamma_f})=\frac{1}{5}\cdot\ell(h_v)=\frac{1}{5}\cdot 2\cdot |OF|=0.3985\dots
\end{equation*}
The same computation holds also for $\gamma_f^2$.
\par
There is exactly one conjugacy class of order $3$ elliptic elements corresponding to the rotation by $2\pi/3$ around the geodesic $v$ in $\mathbb{H}^3$ connecting $O$ and $V$, which we denote by $\gamma_v$. We see that $\gamma_v$ is conjugate to $\gamma_v^{-1}$ by looking at the isotropy group of $O$: this is isomorphic to $A_5$, and every order $3$ element in $A_5$ is conjugate to its inverse. Denoting by $h_v$ a primitive hyperbolic element in $\Gamma^+_{\gamma_v}$, computations analogous to the case of $\gamma_f$ show that
\begin{equation*}
\vol(\Gamma^+_{\gamma_v}\setminus G_{\gamma_v})=\frac{1}{3}\cdot\ell(h_v)=\frac{1}{3}\cdot 2\cdot |OV|=1.2685\dots
\end{equation*}
\bigskip

The case of the order 2 elements is somewhat more complicated. The isotropy groups of $E$ and $F$ are isomorphic to $D_{10}$, and therefore the three edges of $\mathsf{T}$ labeled with $2$ are the image in $\SW/A_5$ of fixed point set of a single order $2$ elliptic element.
Therefore, there is only one conjugacy class of such elements, given by the $\pi$ rotation around the geodesic $\gamma_e$ in $\mathbb{H}^3$ connecting $E$ and $O$. Denoting by $h_e$ a primitive hyperbolic element in $\Gamma^+_{\gamma_e}$, we have
\begin{equation*}
\ell(h_e)=2(|OE|+|EF|+|FV|)=7.5836\dots
\end{equation*}
Furthermore:
\begin{enumerate}
\item[(a)] the centralizer $G_{\gamma_e}$ has an extra connected component corresponding to hyperbolic elements with axis $e$ that switch the endpoints.
\item[(b)]The centralizer $\Gamma^+_{\gamma_e}$ contains the group generated by $\gamma_e$ and $h_e$ as an index $2$ subgroup; more specifically, it contains an extra involution commuting with $\gamma_e$.
This is given by another order two elliptic element in the isotropy group of $O$, corresponding to the fact that order $2$ elements in $A_5$ have centralizer isomorphic to the Klein four group.\end{enumerate}
Putting things together, we obtain
\begin{equation*}
\vol(\Gamma^+_{\gamma_e}\setminus G_{\gamma_e})=2\cdot\frac{1}{4} \ell(h_e)=3.7918\dots
\end{equation*}
where the factors of $2$ and $1/4$ take into account (a) and (b) respectively.
\\
\par
The case of a bad hyperbolic element $h$ is simpler, as in this case the quantity $\vol(\Gamma^+_h\setminus G_h)$ is the length $\ell(h)$ divided by the order of the subgroup of elliptics having the same axis at $h$. Only the case of order $2$ elements require some extra thought, and follows from the fact that the extra involution described in (b) above does not commute with the hyperbolic element.

\vspace{0.3cm}
\subsection{Arithmetic description. }\label{arith}
The group $\Gamma^+$ is a tetrahedral group (see Section $4.7.2$ of \cite{MacReid}), and admits the following arithmetic description (we refer the reader to Chapter $8$ of \cite{MacReid} for the relevant notions). Consider the number field $k=\mathbb{Q} \left( \sqrt{-1-2\sqrt{5}} \right)$, which has exactly one complex place. Consider the quaternion algebra $A$ over $k$ ramified exactly at the two real places, and let $\mathcal{O}$ be a maximal order in $A$. Under the complex embedding, we get the inclusion of the norm one elements
\begin{equation*}
\rho:\mathcal{O}^1\hookrightarrow \mathrm{SL}_2(\mathbb{C}),
\end{equation*}
and by projectivizing we obtain the arithmetic group $P\rho(\mathcal{O}^1)\subset \mathrm{PSL}_2(\mathbb{C})$. We have the following:
\begin{equation}\label{isomorphism}
\Gamma^+\cong P\rho(\mathcal{O}^1).
\end{equation}
The proof of this statement can be easily adapted from the analogous result for the tetrahedral group with Coxeter symbol $[3,5,3]$ in Section $11.2.5$ of \cite{MacReid}. First of all, (\ref{prestetr}) readily implies that $\Gamma^+=(\Gamma^+)^{(2)}$. Furthermore, $\Gamma^+$ is arithmetic with invariant trace field and quaternion algebras $k$ and $A$ (see Appendix $13.1$ of \cite{MacReid}).  By Corollary $8.3.3$ in \cite{MacReid}, $\Gamma^+=(\Gamma^+)^{(2)}\subset P\rho(\mathcal{O}^1)$ for some maximal order $\mathcal{O}$ (all of them are conjugate in our case), and the equality follows because the two groups have the same covolume.
\\

\subsubsection{Conjugacy class data for $P\rho(\mathcal{O}^1)$ by arithmetic methods}
While the orbifold corresponding to $\Gamma^+$ is not implemented in the software SnapPy because of its complicated orbifold singularities, the identification (\ref{isomorphism}) makes it feasible to compute the length spectrum of $\Gamma^+$ using techniques from number theory. The method is described in Chapter $30$ of \cite{Voi}, and has been implemented in PARI/GP by Aurel Page \cite{Page}. 

Let $\mathcal{O}$ be an order in a quaternion algebra $A$ over a number field $k$ with ring of integers $\mathbb{Z}_k.$  Given an element $\gamma\in \mathcal{O}^1$, $K = k(\gamma)$ is a quadratic extension of $k$ that embeds in $A.$  Furthermore, $S =K\cap \mathcal{O}$ is a quadratic $\mathbb{Z}_k$-order that embeds in $\mathcal{O}$. The key facts underlying the method are the following:
\begin{itemize}
\item
$\mathcal{O}^1$-conjugacy classes in $\mathcal{O}^1$ having the same characteristic polynomial as $\gamma$ (or equivalently conjugate to $\rho(\gamma)$ in $\mathrm{PSL}_2(\mathbb{C})$) are in bijection with $\mathbb{Z}_k$-algebra embeddings $\varphi:S \hookrightarrow \mathcal{O}$ up to $\mathcal{O}^1$-conjugation.

\item
$\gamma$ is primitive exactly when the embedding $S \hookrightarrow \mathcal{O}$ is \textit{optimal}, i.e. after extending $\varphi$ linearly to $K$, we have $\varphi(K)\cap \mathcal{O}=\varphi(S)$.

\item
the optimal embeddings $S \hookrightarrow \mathcal{O}$ up to $\mathcal{O}^1$-conjugation may be parametrized adelically.
\end{itemize}
These observations allow one to express the multiplicity of a given element $\mathbb{C}\ell(\gamma)$ in the complex length spectrum in terms of the class number of $S$ and purely local information like local embedding numbers; when $\mathcal{O}$ is a maximal order, the latter can be understood in a very explicit form which is directly computable in PARI/GP.

\vspace{0.5cm}
\section{The Selberg trace formula for coexact $1$-forms for $\SW/A_5$}\label{lower}
The explicit Selberg trace formula for coexact $1$-forms on a hyperbolic three-manifold can be readily generalized to the case of orbifolds. The main complication is the evaluation of the terms in the geometric side corresponding to conjugacy classes which are either elliptic or bad hyperbolic. We have the following.
\begin{thm}[Explicit Selberg trace formula for coexact 1-forms on closed hyperbolic 3-orbifolds]\label{selbergorb}
Let $O$ be a closed oriented hyperbolic three-dimensional orbifold, corresponding to a quotient $\mathbb{H}^3/\Gamma$. Denote by
$0<\lambda_1^*\leq \lambda_2^*\leq \cdots$ the spectrum of the Hodge Laplacian on coexact $1$-forms, and set $t_j=\sqrt{\lambda_j^*}$. Let $H$ be an even, smooth, compactly supported, $\R$-valued function on $\R.$  Then the following identity holds:
\begin{align}\label{traceformulaorbifold}
\left( \frac{1}{2} b_1(O) -\frac{1}{2} \right) \widehat{H}(0)+\frac{1}{2}\sum_{j=1}^{\infty}\widehat{H}(t_j) &= \frac{\mathrm{vol}(O)}{2\pi}\cdot \left(H(0)-H''(0) \right) \nonumber \\
&+\sum_{[\gamma]\neq 1}t(\gamma)\cdot\mathrm{vol}(\Gamma_{\gamma}\setminus G_{\gamma}) \cdot \frac{\mathrm{cos}(\mathrm{hol}(\gamma))}{|1-e^{\mathbb{C}\ell(\gamma)}| \cdot |1-e^{-\mathbb{C}\ell(\gamma)}|}H \left( \ell(\gamma) \right),
\end{align}
where: 
\begin{itemize}
\item $\widehat{H}(t) := \int_\mathbb{R} H(x) e^{ix \cdot t} dx$ is the Fourier transform of $H$;
\item $t(\gamma)=\frac{1}{2}$ if $\gamma$ is an elliptic element of order $2$, and is $1$ otherwise.
\end{itemize}
Furthermore, the formula holds even the class of less regular compactly supported functions described in \cite{LL}.
\end{thm}

\begin{remark}
In the trace formula \eqref{traceformulaorbifold}, if $\gamma$ is elliptic, the term $ \frac{\mathrm{cos}(\mathrm{hol}(\gamma))}{|1-e^{\mathbb{C}\ell(\gamma)}| \cdot |1-e^{-\mathbb{C}\ell(\gamma)}|}H \left( \ell(\gamma) \right)$ reduces to $ \frac{\mathrm{cos}(\mathrm{hol}(\gamma))}{|1-e^{i \cdot \mathrm{hol}(\gamma)}| \cdot |1-e^{-i \cdot \mathrm{hol}(\gamma)}|}H(0).$
\end{remark}
\bigskip

Note that if $\gamma$ is a good hyperbolic element, then
\begin{equation*}
\mathrm{vol}(\Gamma_{\gamma}\setminus G_{\gamma})=\ell(\gamma_0)
\end{equation*}
where $\gamma_0$ is a primitive element of which $\gamma$ is a multiple. In particular, Theorem \ref{selbergorb} is a direct generalization of Theorem $0.4$ of \cite{LL}. In the case of the orbifold $\SW/A_5$, we determine these volume terms for the elliptic and bad hyperbolic elements in Section \ref{centr}.
\begin{proof}[Proof of Theorem \ref{selbergorb}]
The proof of the formula for manifolds in \cite{LL} adapts directly to the case of orbifolds, provided that there are no elements of order $2$. In particular, each of the terms in the second line corresponds to a quantity of the form
\begin{equation*}
\mathrm{vol}(\Gamma_{\gamma}\setminus G_{\gamma})\int_{G_{\gamma}\setminus G} f(g^{-1}\gamma g) \frac{dg}{dg_{\gamma}}
\end{equation*}
for some suitable $f$. When $\gamma$ is not an elliptic element of order $2$, $G_{\gamma}$ is the connected group corresponding of hyperbolic elements sharing the axis of $\gamma$ and fixing its endpoints, and we proved in \cite{LL} that the integral term (after additional work) is of the form in the theorem. In the case of an elliptic element of order $2$, $G_{\gamma}$ has an additional connected component corresponding to elements preserving the same axis but exchanging the two points at infinity.  The integrand $g \mapsto f(g^{-1} \gamma g)$ is invariant under the action of the order 2 group $G_\gamma^0 \backslash G_\gamma,$ and so integrating over $G_{\gamma} \backslash G,$ which equals the $G_{\gamma}^0 \backslash G_\gamma$ quotient of $G_\gamma^0 \backslash G,$ introduces the claimed factor of $1/2$.
\end{proof}
\vspace{0.5cm}

\section{Proof of Claim \ref{multprop}}\label{proofmult}

\begin{figure}
  \includegraphics[width=0.8\linewidth]{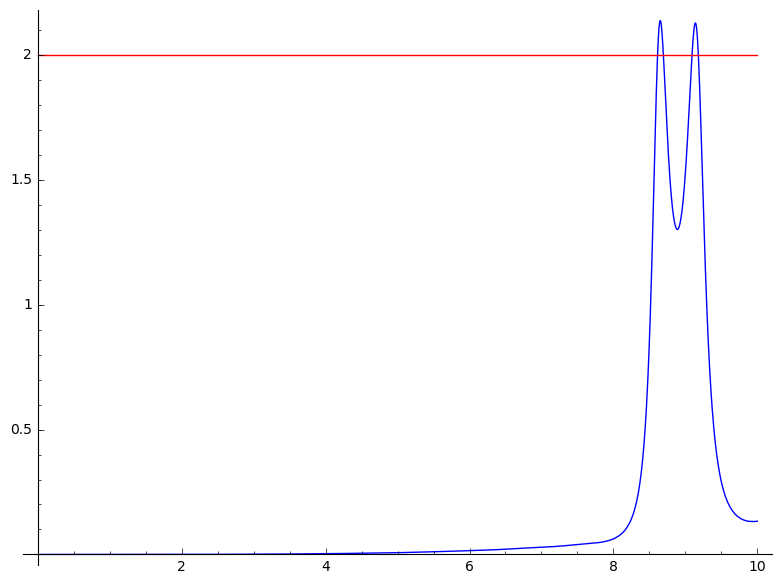}
  \caption{The graph of $t \mapsto J_{8,t}(\SW/A_5)$ for $t \in [0,10]$.}
  \label{SWA5graph}
\end{figure}
In \cite{LL}, we adapted the techniques of Booker and Strombergsson \cite{BS} to provide precise lower bounds for $\lambda_1^*$ of a hyperbolic three manifold $Y$ in terms of its length spectrum. In particular, given the length spectrum of $Y$ up to cutoff $R$, we produced a function $J_{R,t}(Y)$ giving an upper bound to the multiplicity of $t^2$ as an eigenvalue. The same approach (which is based on the Selberg trace formula) can be readily adapted to the case of orbifolds using Theorem \ref{selbergorb} provided we additionally know exactly the list of elliptic and bad hyperbolic elements, together with the associated quantities $\mathrm{vol}(\Gamma_{\gamma}\setminus G_{\gamma})$.
\\
\par
In our case of interest $\SW/A_5$, the length spectrum can be computed via the arithmetic description is \S\ref{arith} using the code of Aurel Page \cite{Page}. The elliptic conjugacy classes were determined in Section \S\ref{centr}. Given an elliptic element of order $n$, there are, up to taking inverses, exactly $n$ primitive bad hyperbolic elements sharing the same axis (with holonomies differing by $2\pi k/n$, for $k=0,1,\dots,n-1$); their length was determined in \S\ref{centr}, and their holonomies are
\begin{itemize}
\item $0,\pi$ when $n=2$;
\item $0, 2\pi/3, 4\pi/3$ when $n=3$;
\item $\pi/5, 3\pi/5, \pi, 7\pi/5, 9\pi/5$ when $n=5$.
\end{itemize}
This can be seen either directly via a geometric argument, or by looking at the length spectrum, as in our specific case these are exactly the holonomies that appear for the lengths of the bad hyperbolics. Finally, the covolumes of the centralizers were determined in \S\ref{centr}.
\\
\par
Using a cutoff $R=8$ we obtain the function $J_{8,t}(\SW/A_5)$ in Figure \ref{SWA5graph}. In particular, $J_{8,t}(\SW/A_5)\leq 1$ for $t\leq 8$, so that $\SW/A_5$ has no eigenvalues $\leq64$. Eigenforms on $\SW/A_5$ correspond to eigenforms on $\SW$ which are invariant under the action of $A_5$ by isometries. In particular, we obtain that for any $\lambda^*$-eigenspace $V_{\lambda^*}$ of $\SW$ for $\lambda^*\leq 64$, there are no copies of the trivial $A_5$-representation. Now $V_{\lambda^*}$ is a finite dimensional representation of the full isometry group $S_5$. The only irreducible representations of $S_5$ of dimension $<4$ are the trivial representation and the sign representation (Section $3.1$ of \cite{FulHar}); both of these representations restrict to the trivial representation of $A_5$. Therefore any $V_{\lambda^*}$ with $\lambda^*\leq64$ is at least $4$-dimensional, and Claim \ref{multprop} follows.

\section{Another tetrahedral orbifold}\label{icosahedron}
We can adapt our discussion to the closely related tetrahedral orbifold $\mathsf{T}'$ with Coxeter symbol $[3,5,3]$. This is obtained from Figure \ref{tetrahedron} by switching the label $3$ to $5$ and viceversa. This is known to be the smallest arithmetic Kleinian group of the form $P\rho(\mathcal{O}^1)$ (see Section $11.7$ in  \cite{MacReid}). The arithmetic description is provided in Section $11.2.5$ of \cite{MacReid}; in particular, it can be identified with $P\rho(\mathcal{O}^1)$ for a maximal order $\mathcal{O}$ in a quaternion algebra over the number field $\mathbb{Q}(\sqrt{3-2\sqrt{5}})$ ramified exactly at the two real places. Our geometric approach from Section \ref{centr} to determine elliptic and bad hyperbolic elements, and their relevant geometric quatities, can be readily adapted to this case, and we obtain lower bounds for the first eigenvalue on coexact $1$-forms as in Figure \ref{fibmodA5graph}.

\begin{figure}
  \includegraphics[width=0.8\linewidth]{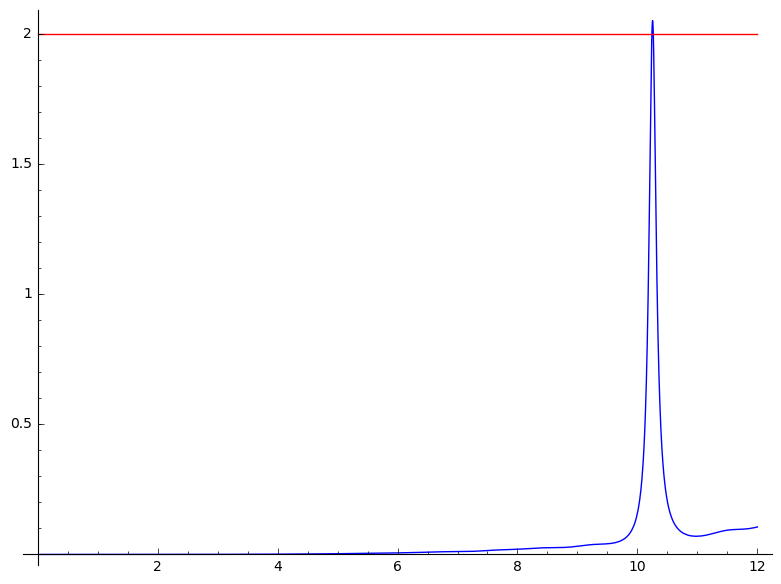}
  \caption{The graph of $t \mapsto J_{8,t}(\mathsf{T}')$ for $t \in [0,10]$.}
  \label{fibmodA5graph}
\end{figure}

\begin{remark}
Both $\SW/A_5$ and $\mathsf{T}'$ admit an orientation reversing isometry $r$ corresponding to the fact that they are the index $2$ subgroups of orientation preserving isometries in a Coxeter group. Geometrically, $r$ is obtained by reflecting along any of the faces of the tetrahedron. This implies that the eigenspaces of the Laplacian on coexact $1$-forms $\Delta$ are even dimensional (which is nicely consistent with Figure \ref{SWA5graph} and \ref{fibmodA5graph}). This is because $\Delta$ acts as $(\ast d)^2$ on coexact $1$-forms, and the action of $r$ on the $\lambda_*$-eigenspace of $\Delta$ exchanges the $\pm\sqrt{\lambda_*}$ eigenspaces of $\ast d$.
\end{remark}

\end{document}